\documentclass[12pt]{amsart}

\usepackage{amssymb, amsthm, amsfonts, amsmath, mathrsfs, latexsym, pstricks-add, graphicx, tikz, hyperref, thmtools, tkz-graph, enumitem}
\usepackage[mathscr]{euscript}
\usepackage{lmodern}
\usepackage[utf8]{inputenc}
\usepackage[T1]{fontenc}
\usepackage{microtype}

\definecolor{Maroon}{RGB}{140,10,0}

\hypersetup{
	colorlinks,
	linkcolor={Maroon},
	citecolor={Maroon},
	urlcolor={Maroon}
}

\newtheorem{theorem}{Theorem}[section]
\newtheorem{lemma}[theorem]{Lemma}
\newtheorem{corollary}[theorem]{Corollary}
\newtheorem{proposition}[theorem]{Proposition}

\theoremstyle{definition}
\newtheorem{definition}[theorem]{Definition}
\newtheorem{XXmpX}[theorem]{Example}
\newenvironment{example}{\pushQED{\qed}\XXmpX}{\popQED\endXXmpX}
\newtheorem{remark}[theorem]{Remark}
\newtheorem{question}[theorem]{Question}

\newcommand{\N}{\mathbb N}

\newcommand{\R}{\mathbb R}

\newcommand{\w}{\omega}
\newcommand{\s}{\sigma}

\usepackage[labelformat=simple]{subcaption}

\newcommand{\SP}{\textit{SP}}
\newcommand{\LmSP}{\textit{LmSP}}
\newcommand{\sLmSP}{\textit{s-LmSP}}

\makeatletter
\@namedef{subjclassname@2010}{%
  \textup{2010} Mathematics Subject Classification}
\makeatother

\numberwithin{equation}{section}

\newcommand{\set}[1]{\left\langle#1\right\rangle}
\newcommand{\setb}[1]{\big\langle#1\big\rangle}
\newcommand{\eps}{\varepsilon}
\DeclareMathOperator{\dist}{dist}
\DeclareMathOperator{\diam}{diam}
\DeclareMathOperator{\Crit}{Crit}
\newcommand{\inv}[2]{\varprojlim \left(#1,#2\right)}

\begin{document}


\baselineskip=17pt


\title{Shadowing, asymptotic shadowing and s-limit shadowing}

\author[C.~Good]{Chris Good}
\address[C.~Good]{School of Mathematics\\
	University of Birmingham\\
	Birmingham, B15 2TT, UK}
\email[C.~Good]{c.good@bham.ac.uk}

\author[P.~Oprocha]{Piotr Oprocha}
\address[P.~Oprocha]{AGH University of Science and Technology, Faculty of Applied
	Mathematics, al.\
	Mickiewicza 30, 30-059 Krak\'ow, Poland
	-- and --
	National Supercomputing Centre IT4Innovations, Division of the University of Ostrava,
	Institute for Research and Applications of Fuzzy Modeling,
	30.\ dubna 22, 70103 Ostrava,
	Czech Republic}
\email[P.~Oprocha]{oprocha@agh.edu.pl}

\author[M.~Puljiz]{Mate Puljiz}
\address[M.~Puljiz]{
	University of Zagreb\\
	Faculty of Electrical Engineering and Computing\\
	Unska 3, 10000 Zagreb, Croatia
}
\email[M.~Puljiz]{puljizmate@gmail.com, mpuljiz@fer.hr}

\date{}

	\begin{abstract}
	We study three notions of shadowing: classical shadowing, limit (or asymptotic) shadowing, and s-limit shadowing. We show that classical and s-limit shadowing coincide for tent maps and, more generally, for piecewise linear interval maps with constant slopes, and are further equivalent to the linking property introduced by Chen in \cite{Chen}.
	
	We also construct a system which exhibits shadowing but not limit shadowing, and we study how shadowing properties transfer to maximal transitive subsystems and inverse limits (sometimes called natural extensions).
	
	Where practicable, we show that our results are best possible by means of examples.
	\end{abstract}


	\maketitle

	It seems that the first time \emph{shadowing} appeared in the literature was in a paper \cite{Sin72} by Sina{\u\i} where it is shown that Anosov diffeomorphisms have shadowing and furthermore that any pseudo-orbit has a unique point shadowing it. This type of shadowing lemma subsequently appears in all sorts of hyperbolic systems, especially in their relation to Markov partitions and symbolic dynamics. We refer the reader to monographs by Palmer \cite{Pal00} and Pilyugin \cite{Pil,PilSak}.
	
	\medskip
	
	Over time different variations of shadowing appeared in the literature driven by different problems people tried to solve using it. Our interest in shadowing comes from its relation to a type of attractors called $\w$-limit sets. Recall that the \emph{$\w$-limit set of a point $x\in X$} is the set of limit points of its forward orbit. It can be shown that each $\w$-limit set is also \emph{internally chain transitive} (ICT), meaning that under the given dynamics and allowing small perturbations one can form a pseudo-orbit between any two points of that set (precise definitions are given in Section \ref{sec:prelim}).
	
	As the ICT condition is more operational, we are interested in systems for which these two notions coincide. It is easy to see that a form of shadowing introduced by Pilyugin et al.\ \cite{ENP97}, called \emph{limit (or asymptotic) shadowing}, suffices.
	
	Over the years it was shown that $\w_f = ICT(f)$ holds for many other systems. Bowen \cite{Bow} proved it for Axiom A diffeomorphisms; and in a series of papers \cite{AndyCon,AndyJuli,AndyPWM,AndyShadow,AndyShift,BarGooOprRai} Barwell, Davies, Good, Knight, Oprocha, and Raines proved it for shifts of finite type and Julia sets for certain quadratic maps, amongst others. It soon became apparent that most of these systems satisfy both the classical and the asymptotic notion of shadowing. This led Barwell, Davies, and Good \cite[{Conjectures 1.2 and 1.3}]{AndyCon} to conjecture that classical shadowing alone will imply $\w_f = ICT(f)$. Recently, Meddaugh and Raines \cite{MeddaughRaines} confirmed this for interval maps\footnote{It is worth noting here that neither classical nor limit shadowing implies the other one --- more on this in Section \ref{sec:shad}.}, and Good and Meddaugh \cite{GooMed} have shown that, for general compact metric $X$, $\w_f = ICT(f)$ if and only if $f$ satisfies Pilyugin's notion of \emph{orbital limit shadowing} \cite{Pil-limit}. In Section \ref{sec:shad} we disprove the other conjecture by constructing a system with shadowing but for which $\w_f \subsetneq ICT(f)$ and thus without limit shadowing. This is the content of Theorem \ref{thm:cntexample}.
	
	{
		\renewcommand*{\thetheorem}{\ref{thm:cntexample}}
		\addtocounter{theorem}{-1}
		\begin{theorem}
			There exists a dynamical system $(X,f)$ on a compact metric space which exhibits shadowing but for which $\w_f \neq ICT(f)$.
		\end{theorem}
	}
	
	Interestingly, Meddaugh and Raines's result did not answer whether for interval maps classical shadowing implies limit shadowing. Trying to resolve this question forms the second line of inquiry in this paper. Sadly, we have not yet been able to provide a definite answer for all interval maps, but a great deal can be said for systems given by piecewise linear maps with constant slopes. Indeed, Chen \cite{Chen} showed that the \emph{linking condition} is strong enough to completely resolve whether such a system has shadowing.
	
	Using the same condition we have been able to show that for these systems shadowing and linking are equivalent to a condition tightly related to limit shadowing, called \emph{s-limit shadowing}. This notion, stronger than either classical or limit shadowing, was introduced in \cite{Sakai} where Sakai extended the definition of limit shadowing to account for the fact that many systems have limit shadowing but not shadowing \cite{KO2,Pil}. Below, we quote the two main results we obtain.
	
	{
		\renewcommand*{\thetheorem}{\ref{thm:shadequiv}}
		\addtocounter{theorem}{-1}
		\begin{theorem}
			Let $f\colon I \to I$ be a continuous piecewise linear map with a constant slope $s>1$. Then the following are equivalent:
			\begin{enumerate}
				\item $f$ has s-limit shadowing,
				\item $f$ has shadowing,
				\item $f$ has the linking property.
			\end{enumerate}
			If furthermore the map is transitive, all of the above are additionally equivalent to:
			\begin{enumerate}
				\setcounter{enumi}{3}
				\item $f$ has limit shadowing.
			\end{enumerate}
		\end{theorem}
	}
	
	{
		\renewcommand*{\thetheorem}{\ref{thm:tent}}
		\addtocounter{theorem}{-1}
		\begin{theorem}
			Let $f_s\colon [0,1]\to [0,1]$ be a tent map with $s\in (\sqrt{2},2]$ and denote its core by $C=[f_s^2(1/2),f_s(1/2)]$. The following conditions are equivalent:
			\begin{enumerate}
				\item $f_s$ has s-limit shadowing,
				\item $f_s$ has shadowing,
				\item $f_s$ has limit shadowing,
				\item $f_s|_C$ has s-limit shadowing,
				\item $f_s|_C$ has shadowing,
				\item $f_s|_C$ has limit shadowing.
			\end{enumerate}
		\end{theorem}
	}
	
	\bigskip
	
	The article is organised as follows. In Section \ref{sec:prelim} we introduce all the basic notions we use. In Section \ref{sec:shadcond} we prove our main results, Theorems \ref{thm:shadequiv} and \ref{thm:tent}. Section \ref{sec:shad} clarifies the distinction between three notions of shadowing: limit, s-limit, and classical shadowing; we also construct a system which exhibits shadowing but not limit shadowing (Theorem \ref{thm:cntexample}). In Section \ref{sec:maximal} we include a few results providing sufficient conditions for shadowing on maximal limit sets of interval maps. Finally, Section \ref{sec:invlim} contains a result on shadowing of the shift map in inverse limit spaces.

	\section{Preliminaries}\label{sec:prelim}
	
	An interval map $f\colon [a,b]\to [a,b]$ is \emph{piecewise monotone} if there are $p_0=a<p_1<\dots <p_k<p_{k+1}=b$
	such that $f|_{[p_i,p_{i+1}]}$ is monotone for $i=0,\ldots,k$; and $f$ is \emph{piecewise linear} if moreover, for each $i$, $f|_{[p_i,p_{i+1}]}(x)=a_i+b_i x$ for some $a_i,b_i\in \R$. If further there is $s\ge 0$ such that $|b_i|=s$ for all $i$, we say that $f$ is \emph{piecewise linear with constant slope $s$}.
	
	A point $x\in [a,b]$ is a \emph{critical point} of $f$ if $x=a$, or $x=b$, or $f$ is not differentiable at $x$, or $f'(x)=0$. The set of critical points of $f$ is denoted by $\Crit(f)$.

	For every $s\in (1,2]$ define a tent map $f_s\colon [0,1]\to [0,1]$ by
	\begin{equation*}
	f_s(x)=
	\begin{cases}
	sx &\text{ if } x\in [0,1/2],\\
	s(1-x) &\text{ if } x\in [1/2,1].
	\end{cases}
	\end{equation*}
	
	\begin{definition}[Pseudo-orbit]
		A sequence $\langle x_0,x_1,x_2,\dots \rangle$ is said to be a \emph{$\delta$-pseudo-orbit} for some $\delta>0$ if $d(f(x_i),x_{i+1})<\delta$ for each $i\in\N_0$. A \emph{finite $\delta$-pseudo-orbit of length $l\ge 1$} is a finite sequence $\langle x_0,x_1,x_2,\dots, x_l\rangle$ satisfying $d(f(x_i),x_{i+1})<\delta$ for $0\le i < l$. We also say that it is a \emph{$\delta$-pseudo-orbit between $x_0$ and $x_l$}.
		
		We say that the sequence $\langle x_0,x_1,x_2,\dots \rangle$ is an \emph{asymptotic pseudo-orbit} if $\lim\limits_{i\to\infty} d(f(x_i),x_{i+1})=0$.
	\end{definition}
	
	\begin{definition}
		A point $z\in X$ is said to \emph{$\eps$-shadow} a sequence $\langle x_0,x_1,x_2,\dots \rangle$ for some $\eps>0$ if $d(x_i,f^{i}(z))<\eps$ for each $i\in\N_0$; and $z$ \emph{asymptotically sha\-dows} the sequence if $\lim\limits_{i\to\infty}d(x_i,f^{i}(z))=0$.
	\end{definition}
	
	\begin{definition}[Shadowing]
		A dynamical system $f\colon X\to X$ is said to have \emph{shadowing} if for every $\eps>0$ there exists a $\delta>0$ such that every $\delta$-pseudo-orbit is $\eps$-shadowed by some point in $X$.
	\end{definition}
	
	If a sequence $\set{x_n}_{n\in \N}$ is a $\delta$-pseudo-orbit and an asymptotic pseudo-orbit then we simply say that it is an \textit{asymptotic $\delta$-pseudo-orbit}.
	Similarly, if a point $z$ $\eps$-traces and asymptotically traces a pseudo-orbit $\set{x_n}_{n\in \N}$, then we say that $z$ \textit{asymptotically $\eps$-traces} $\set{x_n}_{n\in \N}$.
	
	\begin{definition}[s-limit shadowing]
		Let $f\colon X\rightarrow X$ be a continuous map on a compact metric space $X$. We say that \emph{$f$ has s-limit shadowing} if for every $\eps > 0$ there is $\delta > 0$ such that
		\begin{enumerate}
			\item for every $\delta$-pseudo-orbit $\set{x_n}_{n\in \N}\subset X$ of $f$, there is $y \in X$ that $\eps$-shadows $\set{x_n}_{n\in\N}$, and
			\item for every asymptotic $\delta$-pseudo-orbit $\set{z_n}_{n\in \N}\subset X$ of $f$, there is $y \in X$ that asymptotically $\eps$-shadows $\set{z_n}_{n\in\N}$.
		\end{enumerate}
	\end{definition}
	
	\begin{remark}
		Note that s-limit shadowing implies both classical and limit shadowing.
	\end{remark}
	
	\begin{definition}
		Let $f\colon X\to X$ be continuous and let $\eps>0$. \emph{A point $x \in X$ is $\eps$-linked to a point $y \in X$ by $f$} if there exists an integer $m \ge 1$ and a point $z$ such that $f^m(z) = y$ and
		$d(f^j(x),f^j(z))\le\eps$ for $j=0,\ldots, m$.
		
		We say \emph{$x\in X$ is linked to $y\in X$ by $f$} if $x$ is $\eps$-linked to $y$ by $f$ for every $\eps > 0$. \emph{A set $A\subset X$ is linked by $f$} if every $x\in A$ is linked to some $y\in A$ by $f$.
	\end{definition}
	
	\begin{definition}
		Let $f\colon [0,1]\to [0,1]$ be a continuous piecewise monotone map and let $\Crit(f)$ be the finite set of critical points of $f$. We say $f$ has the \textit{linking property} if $\Crit(f)$ is linked by $f$.
	\end{definition}
	
	The following theorem is the main result in \cite{Chen}.
	\begin{theorem}[Chen \cite{Chen}]\label{thm:linking}
		Suppose $f\colon [0,1]\to [0,1]$ is a map that is conjugate to a continuous piecewise linear map with constant slope\footnote{Note that the slope is actually $\pm s$ but throughout the text for definiteness we always take $s$ to be the positive value.} $s>1$. Then $f$ has shadowing if and only if it has the linking property.
	\end{theorem}
	
	In the proof of Proposition~15 in \cite{Chen}, Chen shows the following implication.
	\begin{lemma}[Chen \cite{Chen}]
		Let $f\colon [0,1]\to [0,1]$ be a map with constant slope $s>1$. If $f$ has the linking property then
		there is $\hat{\eps}>0$ such that for every $\eps\in (0,\hat{\eps})$ there is $N=N(\eps)>0$ such that for every $x\in X$ there is an integer $n=n(x,\eps)<N$ such that
		$$
		B(f^n(x), s\eps)\subset f^{n+1}(B(x,\eps)\cap f^{-1}(B_{n-1}(f(x),s^2\eps)))
		$$
		where $B_k(x,\eps)=\{y\in [0,1]: |f^i(x)-f^i(y)|<\eps \text{ for }i=0,\ldots,k\}$.
	\end{lemma}
	
	\medskip
	
	Finally, we define two important notions whose relation to shadowing will soon become apparent.
	
	\begin{definition}[Internally chain transitive sets]
		An $f$-invariant and closed set $A\subseteq X$ is said to be \emph{internally chain transitive} if given any $\delta>0$ there exists a $\delta$-pseudo-orbit between any two points of $A$ that is completely contained inside $A$. We denote by $ICT(f)$ the minimal set containing all ICT subsets of $X$ as its elements. This is a subset of the hyperspace $2^X$ of all closed non-empty subsets of $X$.
	\end{definition}
	
	\begin{definition}[$\w$-limit sets]
		The $\w$-limit set of a point $x\in X$ is the set of limit points of its orbit:
		\begin{equation*}
		\w_f(x)=\bigcap_{i=1}^\infty\overline{\bigcup_{j=i}^\infty \{f^j(x)\}}.
		\end{equation*}
		Each $\w$-limit set is a non-empty, closed, $f$-invariant subset of $X$ (see e.g.\ \cite[{Chapter IV}]{BlockCoppel}). In particular, all $\w$-limit sets belong to $2^X$. Let \begin{equation*}
		\w_f=\{ \w_f(x) \mid x\in X \} \subseteq 2^X
		\end{equation*}
		denote the set of all $\w$-limit sets in the system.
	\end{definition}
	
	It is known (see e.g.\ \cite{BarGooOprRai}) that any $\w$-limit set is also internally chain transitive. We thus have the following inclusion of sets in the hyperspace $2^X$:
	\[
	\w_f \subseteq ICT(f).
	\]
	For some systems this is a strict inclusion and it is not hard to find such examples. It is much more interesting to try to characterise systems in which $\w_f$ and $ICT(f)$ coincide. This would be useful as it is easier to check if a given set is ICT than if it is an $\w$-limit set.

	\section{Shadowing condition}\label{sec:shadcond}
	\begin{lemma}\label{lemma:suf_slim}
		Let $f\colon X \to X$ be a continuous map on the compact metric space $X$. Assume that there exist constants $\lambda\geq 1$ and $\hat{\eps}>0$ such that for every $\eps\in (0,\hat{\eps})$, there exist a positive integer $N=N(\eps)$ and $\eta=\eta(\eps)>0$ such that for each $x\in X$, there exists a positive integer $n=n(x,\eps)\leq N$ satisfying
		\begin{align}\label{eq:dagger}
		f\big(B(f^n(x),\eps+\eta)\big)\subset\{f^{n+1}(y) \colon &d(x,y)\leq \eps,\\
		&d(f^i(x),f^i(y))\leq \lambda \eps, 1\leq i \leq n\}.\nonumber 
		\end{align}
		Then $f$ has s-limit shadowing.
	\end{lemma}
	\begin{proof}
		By \cite[{Lemma~2.4}]{Coven} (see also \cite[{Lemma~2.3}]{Coven}) the assumptions are sufficient for shadowing, therefore it is enough to prove
		that for every sufficiently small $\eps>0$ (we may take $\eps<\hat{\eps}$) there is $\delta>0$ such that every asymptotic $\delta$-pseudo-orbit is asymptotically $\eps$-traced.
		
		Fix any $\eps>0$,
		put $\eps_0=\eps/(3\lambda)$, let $\eta_0=\min \set{\eta(\eps_0),\eps_0}$ and $\delta$ be such that
		$d(f^k(x_0),x_k)<\eta_0$ for any $\delta$-pseudo-orbit $\set{x_i}_{i=0}^k$, where $k=1,2,\ldots, N(\eps_0)$.
		We will show that the $\delta$ is as desired.
		
		Fix any asymptotic $\delta$-pseudo-orbit $\set{x_i}_{i=0}^\infty$.
		Denote $\delta_0=\delta$ and
		for every $m>0$ define $\eps_m=\eps_0/2^m$, let $\eta_m=\min \set{\eta(\eps_m),\eps_m}$  and let $0<\delta_m<\eps_m$ be such that
		$d(f^k(x_0),x_k)<\eta_m$ for any $\delta_m$-pseudo-orbit $\set{x_i}_{i=0}^k$, where $k=1,2,\ldots, N(\eps_m)$.
		
		For any $x\in X$ denote
		$$
		A(x,n,\gamma)=\{y : d(x,y)\leq \gamma, d(f^i(x),f^i(y))\leq \lambda\gamma, 1\leq i \leq n\}.
		$$
		We define integers $m_k, n_k, j_k$, and sets $W_k$ for all $k\geq 0$ as follows:
		$$
		m_0=0, \quad j_0=0, \quad n_0=n(x_0,\eps_{j_0}), \quad W_0=A(x_{m_0},n_0, \eps_0).
		$$
		Next, for $k\geq 1$ we put
		$$
		m_k=m_{k-1}+n_{k-1}=n_0+\ldots+n_{k-1}.
		$$
		By the definition of $j_{k-1}$ the sequence $\set{x_i}_{i=m_k}^\infty$ is an asymptotic $\delta_{j_{k-1}}$-pseudo-orbit.
		If it is also a $\delta_{j_{k-1}+1}$-pseudo-orbit then we put $j_k=j_{k-1}+1$; otherwise we put $j_k=j_{k-1}$.
		Finally, let
		\begin{align*}
		n_k & = n(x_{m_k}, \eps_{j_k}),\\
		W_k & = W_{k-1}\cap f^{-m_k-1}\big(f(A(x_{m_k},n_k,\eps_{j_k}))\big).
		\end{align*}
		
		First, we claim that for every $k\geq 0$ we have
		\begin{equation}\label{eq:star}
		f^{m_{k}+1}(W_k)\subset f\big(A(x_{m_k},n_k,\eps_{j_k})\big). 
		\end{equation}
		We will prove the claim by induction on $k$. For $k=0$ the claim holds just by definition, since
		$$
		f^{m_{0}+1}(W_0)= f(W_0)=f(A(x_{m_0},n_0, \eps_0))
		$$
		Next, fix any $s\geq 0$ and suppose that the claim holds for all $0\leq k \leq s$.
		Since, by definition,
		\begin{align*}
		f^{m_{s+1}+1}(W_{s+1})&=f^{m_{s+1}+1}\big(W_{s}\cap f^{-m_{s+1}-1}\big(f(A(x_{m_{s+1}},n_{s+1},\eps_{j_{s+1}}))\big)\big)\\
		&= f^{m_{s+1}+1}(W_{s})\cap f(A(x_{m_{s+1}},n_{s+1},\eps_{j_{s+1}}))
		\end{align*}
		it remains to prove that
		$$
		f(A(x_{m_{s+1}},n_{s+1},\eps_{j_{s+1}}))\subset f^{m_{s+1}+1}(W_{s}).
		$$
		By the choice of $j_s$ we see that $\set{x_i}_{i=m_s}^\infty$ is a $\delta_{j_s}$-pseudo-orbit and hence
		$$
		d(f^j(x_{m_{s}}),x_{m_{s}+j})<\eta_{j_s} \quad \text{ for every }0\leq j \leq N(\eps_{j_s}).
		$$
		In particular $d(f^{n_s}(x_{m_{s}}),x_{m_{s}+n_s}=x_{m_{s+1}})<\eta_{j_s}$ and thus
		$$B(x_{m_{s+1}},\eps_{j_s})\subset B(f^{n_{s}}(x_{m_s}),\eps_{j_s}+\eta_{j_s}).$$
		By the definition of $n_s$ we also have
		\begin{align*}
		f\big(B(f^{n_{s}}(x_{m_s}),\eps_{j_s}+\eta_{j_s})\big)&\subset \big\{f^{n_s+1}(y) : &&\hspace{-5.5em} d(x_{m_s},y)\leq \eps_{j_s},\\
		&  \qquad \qquad &&\hspace{-5.5em} d(f^i(x_{m_s}),f^i(y))\leq \lambda \eps_{j_s}, 1\leq i \leq n_s \big\}\\
		& =  f^{n_s+1}(A(x_{m_s},n_s,\eps_{j_s})).
		\end{align*}
		Combining the above two observations with the induction assumption we obtain
		\begin{align*}
		f(A(x_{m_{s+1}},n_{s+1},\eps_{j_{s+1}}))&\subset f(B(x_{m_{s+1}},\eps_{j_s})) \; \subset \; f(B(f^{n_{s}}(x_{m_s}),\eps_{j_s}+\eta_{j_s}))\\
		&\subset f^{n_s+1}(A(x_{m_s},n_s,\eps_{j_s}))=f^{n_s}(f^{m_s+1}(W_s))\\
		&=f^{m_{s+1}+1}(W_s).
		\end{align*}
		This completes the induction, so the claim is proved.
		
		Note that since every set $A(x,n,\gamma)$ is closed, we have a nested sequence of closed non-empty sets $W_0\supset W_1\supset \ldots$,
		and therefore there is at least one point  $z\in \bigcap_{k=0}^\infty W_k$.
		We claim that for $k=0$ and $0\leq i \leq m_1$, and also for $k>0$ and each $m_k< i \leq m_{k+1}$, we have
		$$
		d(f^i(z),x_i)\leq 2\lambda \eps_{j_k}.
		$$
		Again we use induction on $k$. First, let $k=0$ and fix any $0\leq i \leq n_0=n(x_0,\eps_0)<N(\eps_0)$. Since $z\in W_0=A(x_0,n_0,\eps_0)$
		we see that
		$d(f^i(z),f^i(x_0))\leq \lambda \eps_0$ and additionally, by the definition of $\delta_{j_0}=\delta_0$, we conclude that
		$d(f^i(x_0),x_i)<\eps_0$. Hence $d(f^i(z),x_i)\leq (\lambda+1)\eps_0\leq 2\lambda\eps_0$. The base case of induction is complete.
		
		Now fix any $k>0$ and any $m_k <  i \leq m_{k+1}$. Denote $t=i-m_k-1$ and observe that $0\leq t<n_k=n(x_{m_k},\eps_{j_k})$. Observe that by \eqref{eq:star},
		$$
		f^i(z)=f^{t}(f^{m_{k}+1}(z))\in f^t(f^{m_{k}+1}(W_k))\subset f^{t+1}(A(x_{m_k},n_k,\eps_{j_k}))
		$$
		and so $d(f^i(z),f^{t+1}(x_{m_k}))<\lambda\eps_{j_k}$. Additionally, by the choice of $j_k$ and $\delta_{j_k}$, we have $d(f^{t+1}(x_{m_k}),x_{m_k+t+1})<\eps_{j_k}$.
		Combining these two inequalities we obtain
		$$
		d(f^i(z),x_i) \leq \lambda\eps_{j_k}+\eps_{j_k}\leq 2\lambda \eps_{j_k}.
		$$
		The claim is proved.
		
		Observe that since $\set{x_i}_{i=0}^\infty$ is an asymptotic pseudo-orbit, the sequence $\set{j_k}_{k=0}^\infty$ is unbounded (i.e.\ $\lim_{k\to \infty} \eps_{j_k}=0$)
		and hence $z$ asymptotically traces $\set{x_i}_{i=0}^\infty$. Additionally,
		$$
		d(f^i(z),x_i) \leq 2\lambda \eps_{j_k}\leq 2\lambda \eps_0 < \eps
		$$
		and so, in fact, $z$ asymptotically $\eps$-traces $\set{x_i}_{i=0}^\infty$, which ends the proof.
	\end{proof}
	
	The equivalence of shadowing and linking for piecewise linear maps with constant slope was demonstrated by Chen \cite{Chen}. In fact, in the proof of \cite[{Proposition 15}]{Chen} it is shown that for such maps the linking property implies the stronger condition \eqref{eq:dagger} with $\lambda= s^2$ and $\eta=(s-1)\eps$ where $s$ is the slope of the map. This observation together with Lemma \ref{lemma:suf_slim} immediately gives the following theorem.
	\begin{theorem}\label{thm:shadequiv}
		Let $f\colon I \to I$ be a continuous piecewise linear map with a constant slope $s>1$. Then the following are equivalent:
		\begin{enumerate}
			\item\label{thm:shadequiv:1} $f$ has s-limit shadowing,
			\item\label{thm:shadequiv:2} $f$ has shadowing,
			\item\label{thm:shadequiv:3} $f$ has the linking property.
		\end{enumerate}
		If furthermore $f$ is transitive, all of the above are additionally equivalent to
		\begin{enumerate}
			\setcounter{enumi}{3}
			\item\label{thm:shadequiv:4} $f$ has limit shadowing.
		\end{enumerate}
	\end{theorem}
	\begin{remark}
		It is worth noting that this is the best one could hope for, as one can construct a (non-transitive) constant slopes piecewise linear map with limit shadowing but without the other properties. Indeed, take a tent map $f=f_s$ for which the critical point $c=1/2$ forms a three-cycle $c\mapsto f(c) \mapsto f^2(c) \mapsto c$. This corresponds to the slope being equal to the golden ratio $s\approx 1.6180$; the map is shown in Figure~\ref{fig:+lim-sh}. Both the map $f$ over $[0,1]$ and the restriction $f|_{[f^2(c),f(c)]}$ have the linking property and thus all the shadowing properties we are considering in this article.
		
		Now take $d=1/(s+s^2) \in(0,f^2(c))$, which is a preimage of the interior fixed point under two iterates. The map $f|_{[d,f(c)]}$ no longer has linking as the prefixed point $d$ is clearly not $\eps$-linked to any of the critical points $\{d,1/2,f(c)\}$ for sufficiently small $\eps$. This implies that this restriction fails shadowing. We claim, however, that it has limit shadowing.
		
		To see this, take any asymptotic pseudo-orbit $\set{x_i}_{i=0}^\infty$ in $[d,f(c)]$. Replace any $x_i\in[d,f^2(c))$ in this sequence with $f^2(c)$ and denote the new sequence in $[f^2(c),f(c)]$ by $\set{y_i}_{i=0}^\infty$. It is easy to check that this is still an asymptotic pseudo-orbit for $f|_{[f^2(c),f(c)]}$, which, as established above, has s-limit shadowing. If we now take $z\in[f^2(c),f(c)]$ which asymptotically shadows the pseudo-orbit $\set{y_i}_{i=0}^\infty$, it is easy to show that it must also asymptotically shadow the original pseudo-orbit $\set{x_i}_{i=0}^\infty$.
		
		To conclude, $f|_{[d,f(c)]}$ is indeed a piecewise linear map with constant slopes which has limit shadowing but fails linking, and consequently does not have shadowing.\qed
	\end{remark}
	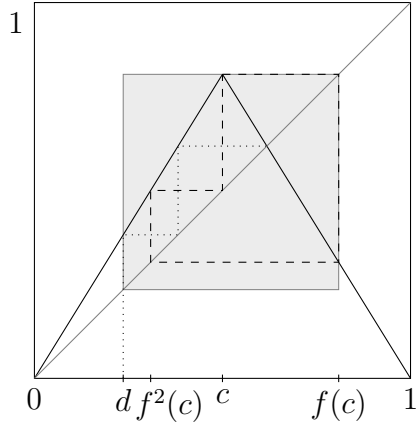
\begin{figure}
		\begin{center}
			\begin{tikzpicture}[scale=5]
			\def\s{1.61803}
			\draw (0,0) rectangle (1,1);
			\draw[gray,line width=.1pt,fill=gray!15] ({1/(\s*(1+\s))},{1/(\s*(1+\s))}) rectangle ({\s/2},{\s/2});
			\draw (0,0) node[anchor=north] {$0$};
			\draw (1,0) node[anchor=north] {$1$};
			\draw (0,1) node[anchor=north east] {$1$};
			\draw (0,0) -- (1/2,{\s/2});
			\draw (1/2,{\s/2}) -- (1,0);
			\draw[smooth,domain=0:1,gray] plot(\x,\x);
			
			\draw[line width=0.4pt,dashed] (1/2,{\s/2})--({\s/2},{\s/2})--({\s/2},{(2*\s-\s*\s)/2}) -- ({(2*\s-\s*\s)/2},{(2*\s-\s*\s)/2}) -- ({(2*\s-\s*\s)/2},{1/2}) --  (1/2,1/2) -- (1/2,{\s/2});
			
			\draw[line width=0.4pt,dotted] ({1/(\s*(1+\s))},0) -- ({1/(\s*(1+\s))},{1/(1+\s)}) -- ({1/(1+\s)},{1/(1+\s)}) -- ({1/(1+\s)}, {\s/(1+\s)}) -- ({\s/(1+\s)}, {\s/(1+\s)});
			
			\draw ({1/(\s*(1+\s))},0) node[anchor=north] {$d$};
			\draw ({1/(\s*(1+\s))},-0.01) -- ({1/(\s*(1+\s))},.01);
			\draw ({1/(2*\s)+.05},0) node[anchor=north] {$f^2(c)$};
			\draw ({1/(2*\s)},-0.01) -- ({1/(2*\s)},.01);
			\draw (1/2,0) node[anchor=north] {$c$};
			\draw (1/2,-0.01) -- (1/2,.01);
			\draw ({\s/2},0) node[anchor=north] {$f(c)$};
			\draw ({\s/2},-0.01) -- ({\s/2},.01);

			\clip(0.,0.) rectangle (1.,1.);
			\end{tikzpicture}
		\end{center}
		\caption{A map with limit shadowing but without the classical shadowing property}\label{fig:+lim-sh}
	\end{figure}
	
	We now proceed to prove Theorem \ref{thm:shadequiv}.
	\begin{proof}[Proof of Theorem \ref{thm:shadequiv}]
		Equivalence of \eqref{thm:shadequiv:2} and \eqref{thm:shadequiv:3} is provided by Theorem~\ref{thm:linking}. If $f$ has linking then the assumptions of Lemma~\ref{lemma:suf_slim} are satisfied, which shows that \eqref{thm:shadequiv:1}
		is a consequence of \eqref{thm:shadequiv:3}. The implication $\eqref{thm:shadequiv:1}\implies \eqref{thm:shadequiv:2}$ is trivial.
		
		In the transitive case the implication $\eqref{thm:shadequiv:4}\implies\eqref{thm:shadequiv:2}$ follows from a result by Kulczycki, Kwietniak, and Oprocha \cite[{Theorem 7.3}]{KKO} whereas the converse $\eqref{thm:shadequiv:1}\implies \eqref{thm:shadequiv:4}$ is immediate by definition.
	\end{proof}
	
	Parry \cite{Par66} showed that any transitive continuous piecewise monotone map of a compact interval is conjugate to a piecewise linear map with a constant slope of the same entropy. By a result of Blokh \cite{Bl82} (for a proof see \cite{BC87}) the entropy $h$ of a transitive interval map is strictly positive (actually $h\ge \log \sqrt{2}$) and thus this slope is $e^h>1$. As each of the four properties in the theorem above is preserved under conjugations, this argument allows us to replace the constant slopes assumption by transitivity.
	
	\begin{corollary}\label{cor:pwmtrans}
		For every transitive continuous piecewise monotone map of a compact interval, all the four properties: s-limit shadowing, limit shadowing, classical shadowing, and linking are equivalent.
	\end{corollary}
	
	It is natural to ask at this point if the characterisation in Theorem \ref{thm:shadequiv} extends to maps with countably many monotone pieces. It turns out that the answer is no, as the following example shows.
	
	\begin{example}
		We construct a piecewise linear map with constant slope $s>1$ and countably many pieces of monotonicity which has the linking property but does not have shadowing.
		
		The key steps of the construction are represented in Figure \ref{fig:countable}. We first construct a nucleus of the map depicted in Figure \ref{fig:nucleus} where each critical point is mapped to another critical point in the rescaled version of the map. To be precise, the critical point denoted by $C_1$ in Figure \ref{fig:nucleus} is mapped to $C_2^+$, the rescaled (by a factor $\mu<1$) copy of the critical point $C_2$ which is in turn mapped to the original point $C_2$. The slope is everywhere the same and is denoted by $s>3$. If we denote by $a$ the length of the first (and third) piece of monotonicity of the nucleus, one easily deduces that the parameters $a,s,\mu$ have to satisfy the following set of equations
		\begin{gather*}
		sa = 1 + \mu(1-a),\\
		s\mu a = \mu + a,\\
		4 s a = s+1.
		\end{gather*}
		This system has a solution $a\approx0.301696, \mu\approx 0.657298, s\approx 4.83598$.
		
		Once the nucleus is constructed, one takes two $\mu$-scaled copies of it and glues them to both ends of the nucleus. Then another two $\mu^2$-scaled copies are added, and so on ad infinitum. This process gives the map depicted in Figure \ref{fig:mapf} which we denote by $f\colon [0,1] \to [0,1]$ after rescaling to the unit interval.
		
		Note that each critical point of $f$ (except $0$ and $1$ which are fixed) is pre-periodic and is eventually mapped onto the cycle $C_1\mapsto C_2^+ \mapsto C_2 \mapsto C_1^- \mapsto C_1$ on the nucleus. Thus $f$ has linking. It is not hard to see that $f$ is also topologically mixing. Namely, its slope exceeds $2$ and so every open interval $U\subset [0,1]$ must eventually cover two critical points (i.e.\ $f^n(U)$ contains two consecutive critical points), which in turn implies that for every $\eps>0$ there is $m>0$ such that $[\eps,1-\eps]\subset f^m(U)$. This proves topological mixing. Unfortunately, it is not completely clear whether $f$ has shadowing, and to ensure this we make a simple modification.
		
		Let $g\colon [-1,1] \to [-1,1]$ be given by
		\begin{equation*}
		g(x)=
		\begin{cases}
		-f(x) & \text{if } x\ge 0,\\
		f(-x) & \text{if } x < 0.
		\end{cases}
		\end{equation*}
		It is clear that $g$ also has linking and constant slopes, but it cannot have shadowing. The reason is that $g$ is transitive but not mixing, as $g^2$ has invariant intervals $[-1,0]$ and $[0,1]$, thus no point can have a dense orbit in $[-1,1]$ under $g^2$. On the other hand, $g^2$ is chain transitive on $[-1,1]$ because for any $\delta$-pseudo-orbit, jumping to the other side of the origin poses no problem. If $g$ had shadowing then so would $g^2$ and this would contradict a result from \cite[{Corollary 6}]{MeddaughRaines} which says that for any interval map with shadowing, any ICT set is also an $\w$-limit set.
	\end{example}
	
	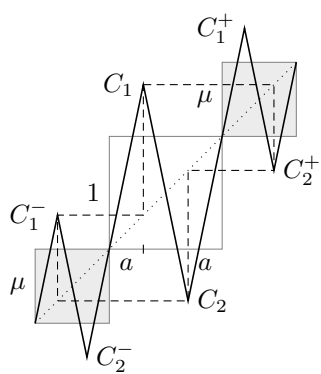
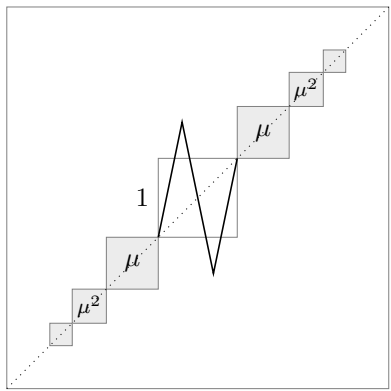
\begin{figure}
		\begin{subfigure}{0.49\textwidth}
			\begin{center}
				\begin{tikzpicture}[scale=1.5]\footnotesize
				\def\m{0.657298}
				\def\s{4.83598}
				\def\a{0.301696}
				
				\draw [gray,line width=.1pt] (0,0) rectangle (1,1);
				\draw [gray,line width=.1pt,fill=gray!15] (1,1) rectangle (1+\m,1+\m);
				\draw [gray,line width=.1pt,fill=gray!15] (0,0) rectangle (-\m,-\m);
				
				\draw[line width=0.6pt] (-\m,-\m)  -- (-\m+\m*\a,-\m+\m*\s*\a) -- (-\m*\a,-\m*\s*\a) -- (\a,\s*\a) -- (1-\a,1-\s*\a)  -- (1+\m*\a,1+\m*\s*\a) -- (1+\m-\m*\a,1+\m-\m*\s*\a) -- (1+\m,1+\m);
				\draw[dotted] (-\m,-\m) -- (1+\m, 1+\m);
				
				\draw (\a/2,0) node[anchor=north] {$a$};
				
				\draw (1-\a/2,0) node[anchor=north] {$a$};
				
				\draw[densely dashed] (\a,\s*\a) -- (\s*\a,\s*\a) -- (1+\m-\m*\a,1+\m-\m*\s*\a) -- (1-\a,1-\a) -- (1-\a,1-\s*\a) -- (-\m+\m*\a,1-\s*\a) -- (-\m+\m*\a, \m*\s*\a-\m) -- (\a,\m*\s*\a-\m) -- (\a,\s*\a);
				
				\draw (-\m,-\m/2) node[anchor=east] {$\mu$};
				\draw (0,1/2) node[anchor=east] {$1$};
				\draw (1,1+\m/2) node[anchor=east] {$\mu$};
				
				\draw (\a,\s*\a) node[anchor=east] {$C_1$};
				\draw (1-\a,1-\s*\a) node[anchor=west] {$C_2$};
				\draw (\a,-0.04) -- (\a,0.04);
				
				\draw (1+\a*\m,1+\a*\m*\s) node[anchor=east] {$C_1^+$};
				\draw (1+\m-\m*\a,1+\m-\m*\s*\a) node[anchor=west] {$C_2^+$};
				\draw (-\m+\m*\a,\a) node[anchor=east] {$C_1^-$};
				\draw (-\m*\a,-\m*\a*\s) node[anchor=west] {$C_2^-$};
				\end{tikzpicture}
			\end{center}
			\caption{The nucleus of the map}\label{fig:nucleus}
		\end{subfigure}
		\hfil
		\begin{subfigure}{0.49\textwidth}
			\begin{center}
				\begin{tikzpicture}[scale=1.05]\footnotesize
				\def\m{0.657298}
				\def\s{4.83598}
				\def\a{0.301696}
				\def\l{1.91798705581}
				
				\draw[gray,line width=.1pt] (-\l,-\l) rectangle (1+\l,1+\l);
				\draw[gray,line width=.1pt] (0,0) rectangle (1,1);
				\draw [gray,line width=.1pt,fill=gray!15] (1,1) rectangle (1+\m,1+\m);
				\draw [gray,line width=.1pt,fill=gray!15] (1+\m,1+\m) rectangle (1+\m+\m^2,1+\m+\m^2);
				\draw [gray,line width=.1pt,fill=gray!15] (1+\m+\m^2,1+\m+\m^2) rectangle (1+\m+\m^2+\m^3,1+\m+\m^2+\m^3);
				
				\draw [gray,line width=.1pt,fill=gray!15] (0,0) rectangle (-\m,-\m);
				\draw [gray,line width=.1pt,fill=gray!15] (-\m,-\m) rectangle (-\m-\m^2,-\m-\m^2);
				\draw [gray,line width=.1pt,fill=gray!15] (-\m-\m^2,-\m-\m^2) rectangle (-\m-\m^2-\m^3,-\m-\m^2-\m^3);
				
				\draw[line width=0.6pt] (0,0) -- (\a,\s*\a) -- (1-\a,1-\s*\a)  -- (1,1);
				\draw[dotted] (-\l,-\l) -- (1+\l, 1+\l);
				
				\draw (-\m/2,-\m/2) node {$\mu$};
				\draw (0,1/2) node[anchor=east] {$1$};
				\draw (1+\m/2,1+\m/2) node {$\mu$};
				\draw (-\m-\m^2/2,-\m-\m^2/2) node {\tiny$\mu^2$};
				\draw (1+\m+\m^2/2,1+\m+\m^2/2) node {\tiny$\mu^2$};
				\end{tikzpicture}
			\end{center}
			\caption{The map $f$ with the linking property}\label{fig:mapf}
		\end{subfigure}
		\caption{Construction of the map which has linking but not shadowing}\label{fig:countable}
	\end{figure}
	
	\begin{example}
		By \cite{KosO} the standard tent map $f_2$ can be perturbed to a map $f\colon [0,1]\to [0,1]$ such that:
		\begin{enumerate}
			\item $f$ has shadowing and is topologically mixing,
			\item the inverse limit of $[0,1]$ with $f$ as a unique bonding map is the pseudo-arc, and hence $f$ has infinite topological entropy (see \cite{Mouron}).
		\end{enumerate}
		By the above $f$ is not conjugate to a piecewise-linear map with constant slope, since all these maps have finite entropy.
		
		The answer to the question whether $f$ has s-limit shadowing is unknown to the authors.
	\end{example}

	\begin{theorem}\label{thm:tent}
		Let $f_s\colon [0,1]\to [0,1]$ be a tent map with $s\in (\sqrt{2},2]$ and denote its core by $C=[f_s^2(1/2),f_s(1/2)]$. The following conditions are equivalent:
		\begin{enumerate}
			\item\label{tm:slim} $f_s$ has s-limit shadowing,
			\item\label{tm:sh} $f_s$ has shadowing,
			\item\label{tm:lim} $f_s$ has limit shadowing,
			\item\label{tm:slim_c} $f_s|_C$ has s-limit shadowing,
			\item\label{tm:sh_c} $f_s|_C$ has shadowing,
			\item\label{tm:lim_c} $f_s|_C$ has limit shadowing.
		\end{enumerate}
	\end{theorem}
	\begin{proof}
		The equivalence $\eqref{tm:slim}\Longleftrightarrow\eqref{tm:sh}$ is provided by Theorem~\ref{thm:shadequiv}. Similarly, since $f_s|_C$ is a unimodal map with constant slope, we obtain $\eqref{tm:slim_c}\Longleftrightarrow\eqref{tm:sh_c}$. In fact by Theorem~\ref{thm:shadequiv} we know that for $f_s$ and $f_s|_C$, shadowing is equivalent to linking.
		
		Let $c=1/2$ be the unique critical point of $f_s$ in $(0,1)$. Note that $\Crit(f_s)=\{0,c,1\}$ and $\Crit(f_s|_{C})=\{f_s^2(c),c,f_s(c)\}$. As $f_s(1)=0$ and $f_s(0)=0$, both $0$ and $1$ are linked to $0$ by $f_s$. Similarly, both $c$ and $f_s(c)$ are linked to $f_s^2(c)$ by $f_s|_C$. Thus for $s<2$ checking the linking property of $f_s$ (resp.\ $f_s|_C$) boils down to checking whether $c$ is linked to itself (resp.\ whether $f^2_s(c)$ is linked to itself).
		
		But if $c$ is linked to itself then clearly $f_s^2(c)$ is also linked to $c$ and thus to $f_s^2(c)$. For the converse, note that if $s<2$, the only preimage of $f_s^2(c)$ under $f_s|_C$ is $f_s(c)$ and in turn the only preimage of $f_s(c)$ is $c$. It is now straightforward to check that if $f_s^2(c)$ is linked to itself then $c$ must also be linked to itself. Thus we have showed $\eqref{tm:slim}\Longleftrightarrow\eqref{tm:sh}\Longleftrightarrow\eqref{tm:slim_c}\Longleftrightarrow\eqref{tm:sh_c}$ (note that for $s=2$ this is trivially satisfied).
		
		If $f$ has s-limit shadowing then by definition it has limit shadowing, so the implications
		$\eqref{tm:slim}\Longrightarrow\eqref{tm:lim}$ and $\eqref{tm:slim_c}\Longrightarrow\eqref{tm:lim_c}$ are trivially satisfied.
		
		It is proved in \cite{KKO} that if a map is transitive and has limit shadowing then it also has shadowing.
		But it is well known (see e.g.\ \cite[{Remark 3.4.17}]{Brucks}) that each $f_s|_C$ is transitive,
		hence $\eqref{tm:lim_c}\Longrightarrow\eqref{tm:sh_c}$ is also valid.
		
		To close the circle of implications it now remains to show $\eqref{tm:lim}\Longrightarrow\eqref{tm:lim_c}$. To this end we fix any asymptotic pseudo-orbit $\set{x_i}_{i=0}^\infty \subset C$
		and let $z\in [0,1]$ be a point which asymptotically traces it under the action of $f_s$. Observe that the set $\Lambda=\omega(z,f_s)$ is $f_s$-invariant and thus $\Lambda\subset C$. Note that for $\sqrt{2}<s<2$ (again for $s=2$ the implication $\eqref{tm:lim}\Longrightarrow\eqref{tm:lim_c}$ trivially holds) the $f_s$-invariant set $\Lambda$ cannot be contained in $\{f_s^2(c), f_s(c)\}$ but must intersect the interior of $C$. Hence there exists an integer $n\ge 0$ such that $f^n(z)\in C$ and so $f^{n+j}(z)\in C$ for every $j\geq 0$.
		If we fix any point $y\in C$ such that $f^n(y)=f^n(z)$ then clearly $y$ asymptotically traces $\set{x_i}_{i=0}^\infty$. This completes the proof.
	\end{proof}
	
	\section{How limit shadowing, s-limit shadowing, and classical shadowing relate to each other}\label{sec:shad}
	In this section we describe the distinction between three notions of shadowing: limit ($\LmSP$), s-limit ($\sLmSP$), and classical shadowing property ($\SP$).
	
	\medskip
	
	Pilyugin \cite[{Theorem 3.1.3}]{Pil} showed that for circle homeomorphisms, $\SP$ implies $\LmSP$. 
	In fact, he gave efficient characterisations of both $\SP$ and $\LmSP$ for orientation preserving circle homeomorphisms which fix a nowhere dense set containing at least two points (see \cite[{Theorems 3.1.1 and 3.1.2}]{Pil}).
	
	Pilyugin's results roughly state that such a homeomorphism with a hyperbolic (either repelling or attracting) fixed point has $\LmSP$, and it further has $\SP$ if the repelling and attracting fixed points alternate. We abstain from stating the full characterisation precisely and instead refer the interested reader to \cite{Pil}.
	
	Using this, it is not hard to construct maps that land in areas denoted by (a), (b), and (c) in the Venn diagram in Figure \ref{fig:venn}. The corresponding graphs are depicted in Figure \ref{fig:3maps}. (Note that the circle is represented as the interval $[-1,1]$ with endpoints identified.)
	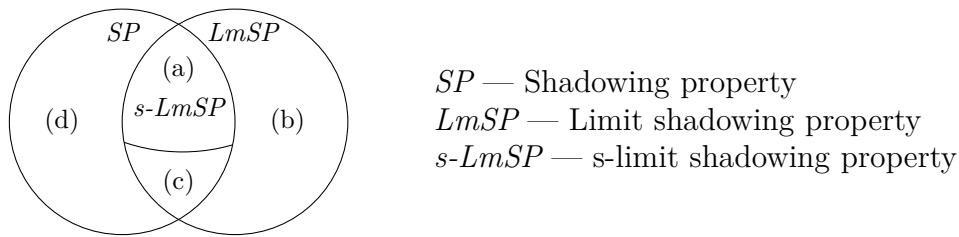
\begin{figure}[ht]
		\def\firstcircle{(0,0) circle (1.5cm)}
		\def\secondcircle{(.8cm,2cm) circle (2.4cm)}
		\def\thirdcircle{(0:1.5cm) circle (1.5cm)}
		\begin{tikzpicture}
		\begin{scope}\footnotesize
		\draw (0.75,0.7) node {(a)};
		\draw (2.2,0) node {(b)};
		\draw (0.75,-0.8) node {(c)};
		\draw (-0.8,0) node {(d)};
		\draw \firstcircle node[xshift=0cm,yshift=1.2cm] {$\SP$};
		\draw \thirdcircle node [xshift=0.1cm,yshift=1.2cm] {$\LmSP$};
		\clip \firstcircle;
		\clip \thirdcircle;
		\draw \secondcircle node [xshift=-0.05cm, yshift=-1.8cm] {$\sLmSP$};
		\end{scope}
		\begin{scope}[xshift=4cm]
		\draw (0cm,0.5cm) node[right] {$\SP$ --- Shadowing property};
		\draw (0cm,0cm) node[right] {$\LmSP$ --- Limit shadowing property};
		\draw (0cm,-0.5cm) node[right] {$\sLmSP$ --- s-limit shadowing property};
		\end{scope}
		\end{tikzpicture}\caption{Shadowing properties}\label{fig:venn}
	\end{figure}
	
	\medskip
	
	The map in Figure \ref{fig:a} has exactly one repelling fixed point (at $\pm 1$) and one attracting fixed point (at $0$). They are alternating and thus this circle homeomorphism has $\SP$ and $\LmSP$. It trivially has $\sLmSP$ as any point other than $\pm 1$ is attracted to $0$, and therefore for $0<\eps<1/2$ any orbit that $\eps$-shadows an asymptotic $\delta$-pseudo-orbit (where $\delta=\delta(\eps)$ is implied by $\SP$) in fact asymptotically shadows it.
	
	The map in Figure \ref{fig:b} has a repelling fixed point at $\pm 1$ and so it has $\LmSP$. But it does not have $\SP$ because the fixed point at $0$ is non-hyperbolic.
	
	The graph in Figure \ref{fig:c} is an extension to the interval $[-1,1]$ of the map given by Barwell, Good, and Oprocha \cite[{Example 3.5}]{AndyShadow}. Our map is given by
	\begin{equation*}
	\Phi(x)= \begin{cases}
	x+\frac{1}{2\pi\sqrt{2}} x\sin(2 \pi \ln|x|) & \text{if } x\in[-1,0),\\
	x^3 & \text{if } x\in[0,1].\\
	\end{cases}
	\end{equation*}
	One readily checks that this is indeed a strictly increasing map on $[-1,1]$. It has a sequence of fixed points converging to the fixed point at $0$ and each of them is hyperbolic (by convention, $0$ is also hyperbolic as the limit of hyperbolic fixed points). The repelling and attracting fixed points alternate, thus this homeomorphism (call it $\phi$) has both $\SP$ and $\LmSP$. However, it does not have $\sLmSP$, and the argument is essentially the one in \cite{AndyShadow}. We present it below for completeness.
	
	For $\eps=1/2$ and any small $\delta>0$ one can find a fixed point $x_\delta\in (0,\delta)$ and an integer $N\in\N$ large enough such that
	\[
	\langle 1/2, \phi(1/2), \dots, \phi^N(1/2),0,x_\delta,x_\delta,x_\delta,\dots  \rangle
	\]
	is an asymptotic $\delta$-pseudo-orbit. Any point that could potentially $\eps$-shadow this pseudo-orbit would have to be in $(0,1)$. The orbit of such a point would eventually converge to $0$ and would not asymptotically shadow the pseudo-orbit above.
	
	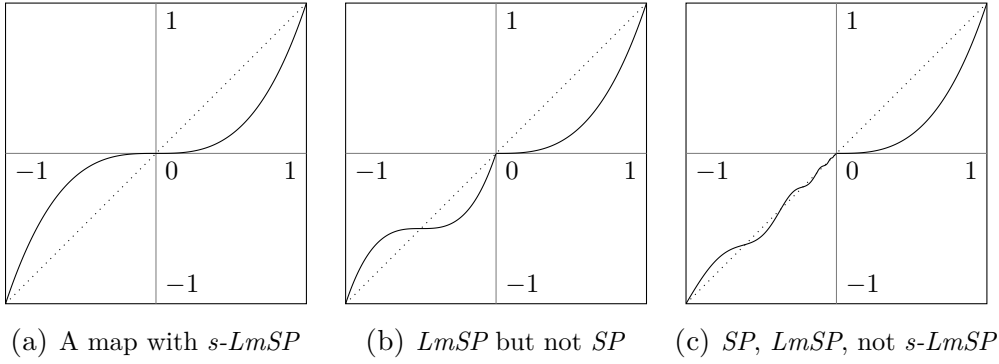
\begin{figure}
		\begin{subfigure}{0.32\textwidth}
			\begin{center}
				\begin{tikzpicture}[scale=2]
				\draw (-1,-1) rectangle (1,1);
				\draw[gray] (-1.,0.) -- (1.,0.);
				\draw[gray] (0.,-1.) -- (0.,1.);
				\footnotesize{
					\draw (0,0) node[anchor=north west] {$0$};
					\draw (1,0) node[anchor=north east] {$1$};
					\draw (0,1) node[anchor=north west] {$1$};
					\draw (0,-1) node[anchor=south west] {$-1$};
					\draw (-1,0) node[anchor=north west] {$-1$};}
				\clip(-1.,-1.) rectangle (1.,1.);
				\draw[smooth,samples=200,domain=-1:1] plot(\x,\x^3);
				\draw[smooth,domain=-1:1,dotted] plot(\x,\x);
				\end{tikzpicture}
			\end{center}
			\caption{\smaller A map with $\sLmSP$}\label{fig:a}
		\end{subfigure}
		\hfil
		\begin{subfigure}{0.32\textwidth}
			\begin{center}
				\begin{tikzpicture}[scale=2]
				\draw (-1,-1) rectangle (1,1);
				\draw[gray] (-1.,0.) -- (1.,0.);
				\draw[gray] (0.,-1.) -- (0.,1.);
				\footnotesize{
					\draw (0,0) node[anchor=north west] {$0$};
					\draw (1,0) node[anchor=north east] {$1$};
					\draw (0,1) node[anchor=north west] {$1$};
					\draw (0,-1) node[anchor=south west] {$-1$};
					\draw (-1,0) node[anchor=north west] {$-1$};}
				\clip(-1.,-1.) rectangle (1.,1.);
				\draw[smooth,samples=200,domain=-1.0:0] plot(\x,{((2*\x+1)^3-1)/2});
				\draw[smooth,samples=200,domain=0:1] plot(\x,\x^3);
				\draw[smooth,samples=2,domain=-1:1,dotted] plot(\x,\x);
				\end{tikzpicture}
			\end{center}
			\caption{\smaller $\LmSP$ but not $\SP$}\label{fig:b}
		\end{subfigure}
		\hfil
		\begin{subfigure}{0.32\textwidth}
			\begin{center}
				\begin{tikzpicture}[scale=2]
				\draw (-1,-1) rectangle (1,1);
				\draw[gray] (-1.,0.) -- (1.,0.);
				\draw[gray] (0.,-1.) -- (0.,1.);
				\footnotesize{
					\draw (0,0) node[anchor=north west] {$0$};
					\draw (1,0) node[anchor=north east] {$1$};
					\draw (0,1) node[anchor=north west] {$1$};
					\draw (0,-1) node[anchor=south west] {$-1$};
					\draw (-1,0) node[anchor=north west] {$-1$};}
				\clip(-1.,-1.) rectangle (1.,1.);
				\draw[smooth,samples=200,domain=-1.0:0] plot(\x,{(\x)+1.0/(sqrt(2.0)*2.0*pi)*(\x)*sin((ln(abs((\x))))*360)});
				\draw[smooth,samples=200,domain=0:1] plot(\x,\x^3);
				\draw[smooth,samples=2,domain=-1:1,dotted] plot(\x,\x);
				\end{tikzpicture}
			\end{center}
			\caption{\smaller $\SP$, $\LmSP$, not $\sLmSP$}\label{fig:c}
		\end{subfigure}
		\caption{Circle homeomorphisms with various shadowing properties}\label{fig:3maps}
	\end{figure}
	
	\medskip
	
	We have already said that Pilyugin ruled out the possibility of a circle homeomorphism in the region marked (d) in Figure \ref{fig:venn}. If one drops the bijectivity requirement and asks only for continuity, the answer seems to be unknown. Until recently it was unknown whether such a system exists on any compact metric space.
	
	To see why this might be interesting, recall that any $\w$-limit set is also an ICT set. As any ICT set can be traced in the limit by an asymptotic pseudo-orbit, for systems with $\LmSP$, ICT sets and $\w$-limit sets coincide. If $\SP$ necessarily implied $\LmSP$ then this would give a positive resolution to a problem posed by Barwell, Davies, and Good \cite[{Conjectures 1.2 and 1.3}]{AndyCon}. They asked if $ICT(f)=\w_f$ for every tent map (or more generally any dynamical system on a compact metric space) with $\SP$.
	
	The answer to \cite[{Conjecture 1.2}]{AndyCon} was given by Meddaugh and Raines \cite{MeddaughRaines}. They showed that the set $ICT(f)$ must be closed in $2^X$, the hyperspace of closed non-empty subsets of $X$ furnished with the Hausdorff distance. Assuming $\SP$ they further showed that $ICT(f)$ is the closure of $\w_f$ in $2^X$:
	\[
	\overline{\w_f} = ICT(f).
	\]
	When combined with a result by Blokh, Bruckner, Humke, and Smítal \cite{omegaInt} that $\w_f$ is already closed for any continuous map $f\colon [0,1]\to [0,1]$, this yields the positive answer for all interval maps.
	
	Their resolution did not, however, resolve whether $\SP$ implies $\LmSP$ even for interval maps. Indeed, we believe that this question is still open.
	
	\begin{question}
		Is it true that for interval maps, $\SP$ implies $\LmSP$?
	\end{question}
	
	Surprisingly, the answer to \cite[{Conjecture 1.3}]{AndyCon} is negative, as we now show. We shall construct a system (Theorem \ref{thm:cntexample}) on the Cantor set with $\SP$ but for which $ICT(f)\neq\w_f$. This in particular implies that this system does not have $\LmSP$ and thus lies in region (d). Independently, Gareth Davies found a similar example which remains unpublished.
	
	\subsection{A system in \texorpdfstring{$\SP\cap (\LmSP)^c$}{SP \textbackslash LmSP}}
	
	Let $\mathcal{A}$ be an alphabet, i.e.\ a finite discrete set of symbols. Recall that a \emph{word} over $\mathcal{A}$ is a finite sequence of elements in $\mathcal{A}$. If one can find a finite collection of words $W$ such that a shift space $X$ is precisely the set of sequences in which no word from $W$ appears, then the shift space is said to be of \emph{finite type}. Walters \cite{Walters} showed that the shifts of finite type are precisely those shift spaces with the shadowing property.
	
	\begin{proposition}[Walters \cite{Walters}]\label{prop:Walters}
		A shift space over a finite alphabet is of finite type if and only if it has shadowing.
	\end{proposition}
	
	We can now proceed with the construction. We fix $\mathcal{A}=\{0,1\}$ and we let $\Sigma_2$ be the full one-sided shift over $\mathcal{A}$. Now for each $k\in\N_0$ we set
	\begin{align*}
	X_k &=\{\xi\in\Sigma_2 \mid \text{any two $1$s in $\xi$ are separated by at least $(k+1)$ $0$s} \},\\
	X_\infty &= \{\xi \in\Sigma_2 \mid \xi \text{ has at most one symbol $1$}\}.
	\end{align*}
	Note that each $X_k$ is in fact a shift space of finite type where the set of forbidden words is exactly
	\[
	\{1\!\underbrace{0\dots0}_{l\text{ zeros}}\!1 \mid 0 \le l \le k \}=\{11, 101, 1001, \dots , 1\!\underbrace{0\dots0}_{k\text{ zeros}}\!1\}.
	\]
	We also set $N=\{{1}/{2^k} \mid k\in\N\cup\{0,\infty\}\}$ where ${1}/{2^\infty}=0$ by convention. The topology on $N$ is taken to be inherited from the real line, and the observant reader might realise that $N$ and $X_\infty$ are in fact homeomorphic. The space $N\times \Sigma_2$ is a compact metric space equipped with the $\max$-distance
	\begin{equation*}
	d ((a_1,\xi_1),(a_2,\xi_2))= \max\{|a_1 - a_2|,d_{\Sigma_2}(\xi_1,\xi_2)\},
	\end{equation*}
	where $d_{\Sigma_2}$ is the standard metric on the full shift space $\Sigma_2$.
	
	On $N\times \Sigma_2$ we define a continuous map $f$ as the product of the identity on $N$ and the shift map $\s$ on $\Sigma_2$:
	\begin{equation*}
	f(a,\xi)=(a,\sigma(\xi)).
	\end{equation*}
	This is easily seen to be continuous.
	Finally, we take
	\begin{equation*}
	X=\{(a,\xi) \in N\times \Sigma_2 \mid a=\frac{1}{2^k} \text{ and } \xi \in X_k \text{, for some $k\in \N\cup\{0,\infty\}$}\},
	\end{equation*}
	or equivalently
	\begin{equation*}
	X \; = \; \{0\}\times X_\infty \; \cup \; \bigcup_{k=0}^\infty \big\{{1}/{2^k}\big\}\times X_k.
	\end{equation*}
	This is clearly an $f$-invariant subset of $N\times\Sigma_2$. Below we show that $X$ is also closed and that $f$ restricted to $X$ provides the counter-example we have been looking for.
	
	\begin{theorem}\label{thm:cntexample}
		$(X,f)$ is a dynamical system on a compact metric space which exhibits shadowing but for which $\w_f \neq ICT(f)$.
	\end{theorem}
	
	Let us briefly describe the idea behind the construction. The map $f$ on each space $\big\{{1}/{2^k}\big\}\times X_k$ is conjugate to a shift of finite type and hence, by Proposition~\ref{prop:Walters}, $f$ has shadowing on those subspaces. The space $\{0\}\times X_\infty$ on the other hand is not of finite type and does not have shadowing. In the construction we exploit the fact that the sequence of spaces $\set{\big\{{1}/{2^k}\big\} \times X_k}_{k\in\N_0}$ converges to $\{0\}\times X_\infty$ in the hyperspace $2^X$ as $k\to\infty$. This allows us to shadow pseudo-orbits in the subspace $\{0\}\times X_\infty$ using real orbits in the space $\big\{{1}/{2^k}\big\}\times X_k$ for $k$ large enough. In this way we succeed (Lemma \ref{lm:shadowing}) to impose shadowing on $f$ in the whole space by having shadowing on a family of proper subspaces approximating $X$.
	
	It remains to check that $\w_f \neq ICT(f)$ for this system. The counter-example is the set $\{0\}\times X_\infty$ which is not the $\w$-limit set of any point in $X$ (Theorem~\ref{thm:cntexample}). Yet, it is the limit of the sequence of subspaces $\big\{{1}/{2^k}\big\}\times X_k$ as $k\to\infty$, each of which is the $\w$-limit set of a point in $X$. This, when combined with the result of Meddaugh and Raines, implies that $\{0\}\times X_{\infty}$ is an ICT set but not an $\w$-limit set. We shall now proceed to prove these claims.
	
	\begin{lemma}
		$X$	is a closed and hence a compact subset of $N\times \Sigma_2$.
	\end{lemma}
	\begin{proof}
		Let $(a,\xi)$ be a point in $N\times \Sigma_2\setminus X$. If $a=1/2^k>0$, this means that $\xi\in \Sigma_2\setminus X_k$. Since $X_k$ is closed in $\Sigma_2$, there is an open set $V$ around $\xi$ that does not intersect $X_k$. Since $U=\{a\}$ is an open (and closed) set in $N$, $U\times V$ is an open neighbourhood containing $(a,\xi)$ that does not intersect $X$.
		
		If $a=0$ and $\xi\not\in X_\infty$ then there exists a $k\in\N_0$ such that the word $1\!\underbrace{0\dots0}_{k\text{ zeros}}\!1$ occurs somewhere in $\xi$. Take $V$ to be the set of all $0$-$1$ sequences in $\Sigma_2$ which have this word at the same position as $\xi$ does. This set is easily seen to be clopen. It is indeed what is called a cylinder set in $\Sigma_2$ (see e.g.\ \cite{LinMar}). Setting $U=[0,1/2^k)\cap N$ one readily checks that $U\times V$ is an open neighbourhood containing $(a,\xi)$ but not intersecting $X$.
	\end{proof}
	
	\begin{lemma}\label{lm:shadowing}
		$(X,f)$ has shadowing.
	\end{lemma}
	\begin{proof}
		Let $\eps>0$ and additionally assume $\eps<1$. Choose a $k\in\N$ so that $\eps/2\leq 1/2^k <\eps$. Set $\delta = \min \{ \eps/4,\delta_1(\eps),\dots,\delta_k(\eps)\} >0$, where each $\delta_j(\eps)$ for $1\le j \le k$ is a positive number chosen so that every $\delta_j(\eps)$-pseudo-orbit in $X_j$ is $\eps$-shadowed. This can be done by Proposition \ref{prop:Walters}. We claim that for this $\delta$, every $\delta$-pseudo-orbit in $X$ is $\eps$-shadowed by a real orbit.
		
		Let $\langle(a_n,\xi_n)\rangle_{n\in\N_0}$ be a $\delta$-pseudo-orbit in $X$. We distinguish two cases.
		
		\begin{itemize}
			\item[\underline{Case 1.}]
			We first suppose that $a_0>\eps/2$. If $a_1>a_0$ then $a_1\geq 2 a_0$ and hence $|a_1-a_0|>\eps/2>\delta$. On the other hand, if $a_1<a_0$ then $2a_1\leq a_0$ and hence $|a_0-a_1|\geq a_0/2>\eps/4\geq \delta$. Therefore, it must be that $a_1=a_0$ and inductively $a_n=a_0$ for all $n\in \N$. Which means that in this case the whole pseudo orbit is actually contained in the same subspace $\big\{{1}/{2^m}\big\}\times X_m$ where $a_0={1}/{2^m}$.
			
			Clearly $m\leq k$. Since $\delta \le \delta_m(\eps)$, we have that $\langle \xi_n \rangle _{n\in\N_0}$ is a $\delta_m(\eps)$-pseudo-orbit in $X_m$, hence we can choose a point $\xi^*$ that $\eps$-shadows it. But then the point $(a_0,\xi^*)$ clearly $\eps$-shadows the initial pseudo-orbit.
			
			\item[\underline{Case 2.}]
			We now suppose $a_0\leq\eps/2$. A similar argument to the one above shows that $a_n\leq\eps/2$, and hence $a_n\leq 1/2^k$ for all $n\in\N$. Since $(X_n)_{n\in\N_0}$ form a decreasing sequence of sets, each $\xi_n$ is contained in the space $X_k$. The sequence $\langle\xi_n\rangle_{n\in\N_0}$ is a $\delta_k(\eps)$-pseudo-orbit in $X_k$, hence there exists a point $\xi^*$ that $\eps$-shadows it. Again, it is readily checked that $({1}/{2^k}, \xi^*)$ $\eps$-shadows the initial pseudo-orbit.\qedhere
		\end{itemize}
	\end{proof}
	
	\begin{proof}[Proof of Theorem \ref{thm:cntexample}]
		It suffices to note that $\{0\}\times X_\infty$ is an ICT set that is not an $\w$-limit set of any of the points in $X$. If it were an $\w$-limit set of some point $(a,\xi) \in X$, it would have to be that $a=0$. But it is not hard to see that the $\w$-limit set of any point in $\{0\}\times X_\infty$ is the singleton $\{(0,0^\infty)\}$ as each point in $\{0\}\times X_\infty$ is eventually mapped to the fixed point $(0,0^\infty)$. Here by $0^\infty$ we denote the sequence in $\Sigma_2$ consisting only of zeros. Therefore the set $\{0\}\times X_\infty$ is not in $\w_f$.
		
		It remains to be shown that $\{0\}\times X_\infty$ is in $ICT(f)$. To simplify notation we shall instead show that the set $X_\infty$ is ICT under the shift map $\sigma$. This is clearly an equivalent statement. Let $\delta > 0$ and let $\xi$ and $\eta$ be any two points in $X_\infty$. We can always choose $k\in\N$ such that $\s^k(\xi)=0^\infty$. If $\eta=0^\infty$ we are done as
		\[
		\langle \xi, \s(\xi), \dots,\s^k(\xi) = \eta \rangle
		\]
		is a $\delta$-pseudo-orbit between $\xi$ and $\eta$.
		
		If otherwise $\eta = 0^m10^\infty$ for some $m\in\N_0$, choose $n>m$ large enough so that the point $\zeta=0^n10^\infty$ is $\delta$-close to $0^\infty$. Then one can check that
		\[
		\langle \xi, \s(\xi), \dots,\s^k(\xi)=0^\infty, \zeta, \s(\zeta), \dots, \s^{n-m}(\zeta) = \eta \rangle
		\]
		is a $\delta$-pseudo-orbit between $\xi$ and $\eta$.
	\end{proof}
	
	The dynamical system $(X,f)$ in Theorem~\ref{thm:cntexample} is not transitive but it is possible to modify it and produce a transitive system on the Cantor set with $\SP$ but not $\LmSP$.
	
	\begin{theorem}\label{thm:cntexampleTransitive}
		There exists a topologically transitive dynamical system on a compact metric space which exhibits shadowing but for which $\w_f \neq ICT(f)$ and which, in particular, fails limit shadowing.
	\end{theorem}
	
	Below we shall give a sketch of the proof but the details are left to the reader. We shall need the following lemma.
	
	\begin{lemma}
		Let $(X,f)$ be a dynamical system and let $F_1 \subset F_2 \subset \dots \subset X$ be an increasing sequence of closed $f$-invariant subsets converging to $X$, $\overline{\bigcup_{i=1}^\infty F_i} = X$. Further assume that for each $n\in\N$ there exists a continuous map $\pi_n \colon X \to F_n$ which is:
		\begin{enumerate}
			\item non-expanding, i.e.\ $d(\pi_n(x),\pi_n(y))\le d(x,y)$,
			\item commuting with $f$, i.e.\ $\pi_n \circ f = f \circ \pi_n$,
			\item a nearest point projection, i.e.\ $d(x,\pi_n(x))=\min\{ d(x,y) : y \in F_n \}$\\
			(note that such $\pi_n$ is a retraction of $X$ onto $F_n$).
		\end{enumerate}
		If $(F_n,f|_{F_n})$ has shadowing for each $n\in\N$ then so does $(X,f)$.
	\end{lemma}
	\begin{proof}
		Let $\eps>0$. Choose $n\in\N$ large enough so that $F_n$ and $X$ are $\eps/2$-close when measured in Hausdorff distance on $2^X$. As $\pi_n$ is a nearest point projection, this implies that
		\begin{equation}\label{eq:ddager}
		d(x,\pi_n(x))\le \eps/2, \text{ for all } x \in X. 
		\end{equation}
		Now let $\delta=\delta_n(\eps/2)>0$ be provided by the shadowing property in $(F_n,f|_{F_n})$ associated to $\eps/2$. We claim that this $\delta$ suffice.
		
		To this end, let $\set{x_0,x_1,\dots}$ be a $\delta$-pseudo-orbit in $X$. The properties of $\pi_n$ ensure that $\set{\pi_n(x_0), \pi_n(x_1),\dots}$ is still a $\delta$-pseudo-orbit in $F_n$. This pseudo-orbit can be $\eps/2$-shadowed by a point $y\in F_n$. But now using \eqref{eq:ddager} one easily checks that $y$ $\eps$-shadows the original pseudo-orbit $\set{x_0,x_1,\dots}$. This completes the proof.
	\end{proof}
	
	This lemma can be seen as a generalisation of Lemma \ref{lm:shadowing}. Each of the sets $F_n=\bigcup_{k=0}^{n-1} \big\{{1}/{2^k}\big\}\times X_k$ is a disjoint union of $n$ shifts of finite type and therefore has shadowing. It is not hard to check that the projections $\pi_n\colon (a,\xi) \mapsto (\max\{a,1/{2^{n-1}}\},\xi)$ have all the required properties.
	
	\medskip
	
	Let us now look at another way to represent the system from Theorem \ref{thm:cntexample}. Let $a_k=1/3^{k+1}$ and $b_k=1-1/3^{k+1}$ (any two other sequences converging to $0$ and $1$ respectively would work equally well). For each $k\in\N_0$ the system $(\{{1}/{2^k}\}\times X_k,f|_{\{{1}/{2^k}\}\times X_k})$ is naturally isomorphic to a shift of finite type $X_k$ which in turn is isomorphic to a shift of finite type where all the occurrences of symbol $0$ are replaced by symbol $a_k$ and where all the occurrences of symbol $1$ are replaced by $b_k$. Now the disjoint union $F_n=\bigcup_{k=0}^{n-1} \big\{{1}/{2^k}\big\}\times X_k$ can be simply represented as a shift of finite type over symbols $\{a_0,\dots, a_{n-1};b_0,\dots,b_{n-1}\}$ where any word containing $a_k$/$b_k$ and $a_l$/$b_l$ with $k\neq l$ is forbidden and, furthermore, any $b_k$ must be preceded (if possible) and followed by at least $(k+1)$ $a_k$s. The set $X$ is then the closure of the union $\bigcup_{n=1}^\infty F_n$ inside the countable product $\{a_0,\dots;0;b_0,\dots;1\}^\N$ where the base set is inheriting the topology from the interval $[0,1]$. The map providing dynamics is the shift map $\s$.
	
	So far we have only given a different description of the example from Theorem \ref{thm:cntexample}. The advantage being that we have disposed of the first coordinate and are now able to represent our system as a shift space albeit over a countable alphabet. The next step is to make the system transitive. We introduce an extra symbol $2$ and for each $n\in\N$ we let $F_n$ be a shift of finite type over $\{2;a_0,\dots, a_{n-1};b_0,\dots,b_{n-1}\}$ where as before any $b_k$ must be preceded (if possible) and followed by at least $(k+1)$ $a_k$s, and any two symbols $a_k$ and $a_l$ with $k\neq l$ need to be separated by at least $(k+l)$ $2$s. Now taking the closure of $\bigcup_{n=1}^\infty F_n$ inside $\{2;a_0,\dots;0;b_0,\dots;1\}^\N$ gives a transitive example with shadowing where the ICT set coinciding with $X_\infty$ is not an $\w$-limit set for the shift map. We leave the details to the reader.
	
	\medskip
	
	In relation to these results, Good and Meddaugh \cite{GooMed} very recently obtained a characterisation of systems in which $\w_f=ICT(f)$. They show that this is equivalent to another technical property named \emph{orbital limit shadowing}, defined in their paper. Our example thus also shows that a system can exhibit shadowing without having orbital limit shadowing.

	\section{Shadowing on maximal \texorpdfstring{$\w$}{w}-limit sets}\label{sec:maximal}
	
	In Theorem \ref{thm:tent} we saw that for tent maps shadowing implies shadowing on the core. For $s\in(\sqrt{2},2]$ the core of a tent map $f_s$ is its maximal $\w$-limit set. It is an $\w$-limit set as it is clearly an invariant set and the map $f_s$ is transitive on its core (see e.g.\ \cite[Remark 3.4.17]{Brucks}), hence there exists a point whose orbit is dense in the core. The core is a maximal $\w$-limit set as the trajectory of any point whose $\w$-limit set contains the core would eventually have to enter the interior of the core, and therefore its $\w$-limit set would after all have to be contained in the core.
	
	The results we prove below can thus be seen as generalising the implication $\eqref{tm:sh}\implies \eqref{tm:sh_c}$ of Theorem \ref{thm:tent}.

	\begin{theorem}
		Let $f\colon [0,1] \to [0,1]$ be an interval map with shadowing and further assume that there are only finitely many maximal $\w$-limit sets for $f$. Then the restriction of $f$ to any of these maximal $\w$-limit sets also has shadowing.
	\end{theorem}
	\begin{proof}
		Let $A_1,\ldots, A_k$ be the collection of all maximal $\omega$-limit sets for $f$. By a result of Meddaugh and Raines \cite[{Corollary 6}]{MeddaughRaines} we know that for a map with shadowing $\w$-limit sets and internally chain transitive (ICT) sets coincide. It is clear that any two maximal ICT sets must be disjoint, as otherwise their union would again be an ICT set contradicting their maximality. One can also see this directly as follows.
		
		Assume that $A_i\cap A_j\neq\emptyset$ for some $i\neq j$. Then there exists an asymptotic pseudo-orbit $\set{z_i}_{i=1}^\infty$ whose set of accumulation points contains $A_i \cup A_j$. But then, by constructing appropriate periodic pseudo-orbits, it is not hard to see that for every $n$ there is a point $y_n$ such that $B(\omega(y_n,f),1/n)\supset A_i\cup A_j$. Since the space of all $\omega$-limit sets of an interval map is closed by a result of Blokh, Bruckner, Humke, and Sm{\'{\i}}tal \cite{omegaInt}, and since the hyperspace $2^X$ is compact, this would imply that there exists an $\omega$-limit set $A\supset A_i\cup A_j$, which is a contradiction. Thus indeed $A_i\cap A_j=\emptyset$ for $i\neq j$.
		
		Denote $\eta=\min_{i\neq j}\dist(A_i,A_j)/2>0$, 
		fix any $\eps>0$ and let $\delta=\delta(\hat{\eps})$ be provided by the shadowing property of $f$ for $\hat{\eps}=\min\{\eps/2,\eta\}$. 
		Fix any $A_i$ and any finite $\delta$-pseudo-orbit $\set{x_0,\ldots, x_n}\subset A_i$. Since $A_i$ is an $\omega$-limit set, it is internally chain transitive
		and therefore there are $x_{n+1},\ldots, x_{m-1}\in A_i$ such that the sequence $\set{x_0,\ldots, x_{m-1}, x_0}$ is a (periodic) $\delta$-pseudo-orbit. Set $y_i=x_{(i\bmod m)}$ for all $i\ge0$ and let $z$ be a point that $\eps/2$-traces the $\delta$-pseudo-orbit $\set{y_i}_{i=0}^\infty$. The set $\w(z,f)$ must be contained in some maximal $\w$-limit set, and from our choice of $\eta$ it is not hard to see that this must be $A_i$.
		
		Now let $q\in \w(z,f)\subset A_i$ be any limit point of the sequence $\set{f^{nm}(z)}_{n=0}^\infty$. It is readily checked that $q$ $\eps$-traces the pseudo-orbit $\set{y_i}_{i=0}^\infty$ and in particular the first portion of it $\set{x_0,\ldots,x_n}$. This shows that any finite $\delta$-pseudo-orbit consisting of points in $A_i$ can be $\eps$-traced by a point from $A_i$. This shows that $f|_{A_i}$ has the shadowing property which completes the proof.
	\end{proof}
	
	In \cite[{Theorem 5.4}]{Bl95} Blokh proved that for interval maps the set $\w(f)=\bigcup_{x\in X}\w(x,f)$ has a spectral decomposition into a family of maximal $\w$-limit sets:
	$$\w(f) = X_f \cup \bigg(\bigcup_\alpha S_\w^{(\alpha)}\bigg) \cup \bigg(\bigcup_i B_i\bigg)$$
	where $X_f$ is a collection of periodic orbits that are also maximal $\w$-limit set (sets of genus $0$ in Blokh's notation); each of $S_\w^{(\alpha)}$s is a solenoidal limit set and none of them contains a periodic orbit (genus $1$); finally each $B_i$ is a so-called basic set, an infinite maximal limit set containing a periodic orbit (genus $2$). It turns out \cite[{Theorem PM6}]{Bl95} (also \cite{Bl86}) that for piecewise linear maps with a constant slope $s>1$ there are no sets of genus $1$ in this decomposition, and the number of maximal $\w$-limit sets of genus $0$ and $2$ is finite. Combining this with the previous theorem we obtain the following result.
	
	\begin{corollary}
		Let $f\colon [0,1] \to [0,1]$ be a piecewise linear map with constant slope $s>1$ and with shadowing. Then the restriction of $f$ to any of its maximal $\w$-limit sets has shadowing.
	\end{corollary}
	\begin{remark}
		The assumption of maximality here is crucial as the ICT set $W=\bigcup_{i=0}^\infty\{\frac{1}{2^i}\} \cup \{0\}$ is an $\w$-limit set for the full tent map $f_2$. Note that $f_2$ has shadowing but its restriction $f_2|_W$ does not.
	\end{remark}

	\begin{theorem}\label{thm:44}
		Let $f\colon [0,1] \to [0,1]$ be a piecewise monotone map with shadowing. Then the restriction of $f$ to any of its $\w$-limit sets with non-empty interior has s-limit shadowing. 
	\end{theorem}
	\begin{proof}
		We first observe that any such limit set $\w(x)$ must in fact be a cycle of disjoint closed intervals (see e.g.\ \cite[{IV. Lemma 5}]{BlockCoppel}). Namely $\w(x) = \bigcup_{i=0}^{p-1} J_i$ where $p\in\N$, the sets $J_i$ are $p$ pairwise disjoint segments contained in $[0,1]$, and $f(J_i)=J_{(i+1 \bmod p)}$. Furthermore, the map $f^p|_{J_0}\colon J_0 \to J_0$ is transitive (as it has a dense orbit) and it is also piecewise monotone. The result will now follow from Corollary \ref{cor:pwmtrans} if we could only show that $f^p|_{J_0}$ has shadowing. We remind the reader that both shadowing and s-limit shadowing hold for a power $f^p$ of the map $f$ if and only if they hold for the map itself.
		
		Let $\eps>0$ and without loss of generality assume that $\eps<\diam(J_0)/2$. Let $\delta=\delta(\eps)>0$ be provided by the shadowing property of $f^p$ over the whole interval $[0,1]$. We claim that this $\delta$ will suffice. To see this, take a finite $\delta$-pseudo-orbit $\set{x_0,\dots,x_n}$ contained in $J_0$. As $f^p$ is transitive on $J_0$ we can extend this pseudo-orbit to a longer $\delta$-pseudo-orbit $\set{x_{-m}, x_{-m+1},\dots, x_0,\dots,x_n}$ where $x_{-m}\in J_0$ denotes the midpoint of the segment $J_0$. Let $z\in I$ be a point that $\eps$-traces this pseudo-orbit. Since $d(z,x_{-m})<\eps$ it must be that $z\in J_0$. Therefore the $f^p$-orbit of $z$ is actually contained in $J_0$ and the point $f^{mp}(z)$ $\eps$-traces $\set{x_0,\dots,x_n}$ under the action of $f^p$. This completes the proof.
	\end{proof}
	\begin{remark}
		The converse to Theorem~\ref{thm:44} does not hold, as the double tent map $f_{2,2}\colon [-1,1]\to[-1,1]$ shows. It is defined as a symmetric extension of the standard tent map $f_2$:
		\begin{equation*}
		f_{2,2} (x) =
		\begin{cases}
		f_2 (x) &\text{if } x\in[0,1],\\
		-f_2 (-x) &\text{if } x\in[-1,0).
		\end{cases}
		\end{equation*}
		This map is piecewise monotone with constant slope $s=2$ and does not have shadowing as the set of critical points $\{-1,-1/2,1/2,1\}$ is not linked ($1$ is mapped to a fixed point $0$). The restrictions $f_{2,2}|_{[0,1]}$ and $f_{2,2}|_{[-1,0]}$ to its two maximal $\w$-limit sets both have linking (as $0$ is now also a critical point) and thus by Theorem~\ref{thm:shadequiv} both have s-limit shadowing.
	\end{remark}
	
	\section{Inverse limit space}\label{sec:invlim}
	
	\begin{theorem}
		Let $X$ be a compact metric space and $f\colon X \to X$ a continuous onto map. Let $\s$ be the shift map on the inverse limit space $\inv{X}{f}$. Then:
		\begin{enumerate}
			\item $\s$ has shadowing if and only if $f$ has shadowing,
			\item $\s$ has limit shadowing if and only if $f$ has limit shadowing,
			\item $\s$ has s-limit shadowing if and only if $f$ has s-limit shadowing.
		\end{enumerate}
	\end{theorem}
	\begin{proof}
		The first equivalence was proven by Chen and Li \cite[{Theorem 1.3}]{ChLi92}, the second by Gu and Sheng \cite[{Theorem 3.2}]{GuSh06}, and we shall now see that the third one also holds. While the idea is clear, one needs to be careful with indexing everything properly.
		
		First we will show that if $\s$ has s-limit shadowing then $f$ has s-limit shadowing. Without loss of generality we may assume that $\diam (X)=1$. Fix any $\eps>0$  and let $\delta>0$ be assigned to $\eps/2$ by s-limit shadowing of $\s$. Let $N$ be such that $\sum_{i=N}^\infty 2^{-i}<\delta/2$ and let $\gamma \in (0,\frac{\eps}{4N})$ be such that if $x,y\in X$ satisfy $d(x,y)<\gamma$ then $d(f^i(x),f^i(y))<\min\{\frac{\delta}{4}, \frac{\eps}{4N} \}$ for $i=0,\ldots,N$.
		We claim that this $\gamma$ suffices.
		
		Fix any asymptotic $\gamma$-pseudo orbit $\set{x_n}_{n=0}^\infty$ for $f$. There exists a strictly increasing sequence of positive integers $\set{u_k}_{k=1}^\infty$ such that for all $n\geq u_k-N$ we have $d(f^{i+1}(x_n),f^i(x_{n+1}))<\gamma/(k+1)^2$ for $i=0,\ldots,k+N-1$. For technical reasons we put $u_0=0$. Define a sequence $\set{z_n}_{n=0}^\infty\subset X$ by taking any $y\in f^{-N}(x_0)$, putting $x_{-i}=f^{N-i}(y)$ for $0\leq i\leq N$ and next:
		$$
		z_n=\begin{cases}
		x_{n-N} &\text{if } n\leq  u_1,\\
		f^k(x_{n-k-N}) &\text{if } u_k+k\leq n \leq  u_{k+1}+k.
		\end{cases}
		$$
		
		We define a sequence of points $\set{y^{(i)}}_{i=0}^\infty \subset \inv{X}{f}$ by the following rule. For $n\leq N$ and $i\geq 0$ we put $y^{(i)}_n=f^{N-n}(z_i)$. For $n>N$ we distinguish two cases when defining $y^{(i)}$. If $i\in [u_{k}+k,u_{k+1}+k]$ for some $k\geq 0$ and if $n\leq N+k$ then we put $y^{(i)}_n=f^{k-n+N}(x_{i-k-N})$ and in the other case we take any $y^{(i)}_n\in f^{-1}(y^{(i)}_{n-1})$ which is possible, since $f$ is onto.
		Note that if $n\leq N+k$ then we have already defined $y_N^{(i)}=z_i=f^k(x_{i-k-N})$ and we also put $y_{N+1}^{(i)}=f^{k-1}(x_{i-k-N})$ and therefore $f(y^{(i)}_{n+1})=y^{(i)}_n$ for every $n\geq 0$, so with this definition indeed $y^{(i)}\in \inv{X}{f}$.
		We claim that the sequence $y^{(i)}$ constructed by the above procedure is an asymptotic $\delta$-pseudo-orbit for $\s$.
		
		Firstly, we claim that $\set{z_n}_{n=0}^\infty\subset X$ is an asymptotic $\gamma$-pseudo-orbit. Fix any $n\in\N$ and assume that both $n$ and $n+1$ belong to $[u_k+k,u_{k+1}+k]$. Then $d(f(z_n),z_{n+1})=d(f^{k+1}(x_{n-k-N}),f^{k}(x_{n-k-N+1}))<\gamma/(k+1)^2$. In the second case $n=u_{k+1}+k$ we have $d(f(z_n),z_{n+1})=d(f^{k+1}(x_{n-k-N}),f^{k+1}(x_{n-k-N}))=0$. Indeed, the claim holds.
		
		By the choice of $\gamma$ for $n\leq N$ we have
		$$
		d(f(y^{(i)}_n),y^{(i+1)}_{n})=d(f^{N-n}(f(z_i)),f^{N-n}(z_{i+1}))<\delta/4,
		$$
		and so
		\begin{align*}
		d(\s(y^{(i)}),y^{(i+1)})&= \sum_{n=0}^\infty 2^{-n} d(f(y^{(i)}_n),y^{(i+1)}_{n})\leq \sum_{n=0}^N 2^{-n-2}\delta + \sum_{n=N+1}^\infty 2^{-n}\\
		&<\delta/2+\delta/2
		\end{align*}
		showing that $y^{(i)}$ is indeed a $\delta$-pseudo orbit for $\s$. To show that it is an asymptotic pseudo-orbit it suffices to show that $\setb{y^{(i)}_n}_{i=0}^\infty$ is an asymptotic pseudo-orbit for any sufficiently large $n$. To this end, fix any $n>N$. Then if $k$ is such that $k\ge n-N$, then for all $i\in [u_{k}+k,u_{k+1}+k)$ we have $i-k-N\geq u_k-N$, and so
		$$
		d(f(y^{(i)}_n),y^{(i+1)}_{n})=d(f^{k-n+N}(f(x_{i-k-N})),f^{k-n+N}(x_{i-k-N+1}))<\frac{\gamma}{(k+1)^2},
		$$
		while if $i=u_{k+1}+k$ then
		$$
		d(f(y^{(i)}_n),y^{(i+1)}_{n})=d(f^{k-n+N}(f(x_{i-k-N})),f^{k-n+N+1}(x_{i-k-N}))=0.
		$$
		This immediately implies that $\set{y^{(i)}}_{i=0}^\infty$ is an asymptotic pseudo-orbit for $\sigma$.
		
		Now, let $p\in \inv{X}{f}$ be a point asymptotically $\eps/2$-tracing a pseudo-orbit $\set{y^{(i)}}_{i=0}^\infty \subset \inv{X}{f}$. Then by the definition of the metric in $\inv{X}{f}$ the coordinate $p_0\in X$ of $p$ asymptotically $\eps/2$-traces the sequence $\setb{y^{(i)}_0}_{i=0}^\infty \subset X$. Then,by the definition of $\gamma$,
		\begin{align*}
		d(f^{i}(p_0),z_{i+N})&\leq d(f^{i}(p_0),f^N(z_i))+\sum_{j=0}^{N-1}d(f^{N-j}(z_{i+j}),f^{N-j-1}(z_{i+j+1}))\\
		&\leq d(f^{i}(p_0),y^{(i)}_0)+N \frac{\eps}{4N} < 3\eps/4.
		\end{align*}
		The same calculation also yields $\lim_{i\to\infty}d(f^{i}(p_0),z_{i+N})=0$. Furthermore, if $i\in [u_k+k-N,u_{k+1}+k-N]$ then $z_{i+N}=f^k(x_{i-k})$ and so
		\begin{align*}
		d(x_i,z_{i+N})	& = d(x_i,f^N(z_i)) = d(x_i,f^{k}(x_{i-k})) \\
		& \leq \sum_{j=0}^{k-1} d(f^j(x_{i-j}),f^{j+1}(x_{i-j-1}))\\
		& \leq k\frac{\gamma}{(k+1)^2}\leq \frac{\gamma}{k+1}<\eps/4.
		\end{align*}
		The same calculation, in particular the bound $\frac{\gamma}{k+1}$, yields $\lim_{i\to\infty}d(x_i,z_{i+N})=0$. Putting together these two calculations shows that $p_0$ asymptotically $\eps$-traces the pseudo-orbit $\set{x_n}_{n=0}^\infty\subset X$ and so the proof of the first implication is completed.
		
		Next we prove that if $f$ has s-limit shadowing then $\s$ has s-limit shadowing.
		Fix any $\eps>0$ and let $N$ be such that $\sum_{i=N}^\infty 2^{-i}<\eps/2$. There is $\delta>0$ such that if $d(x,y)<\delta$ then $d(f^i(x),f^i(y))<\frac{\eps}{2N}$ for $i=0,\ldots,N$.
		There is $\gamma>0$ such that if $\set{x_n}_{n=0}^\infty$ is an asymptotic $\gamma$-pseudo-orbit then it is $\delta$-traced.
		Let $\set{y^{(i)}}_{i=0}^\infty \subset \inv{X}{f}$ be an asymptotic $2^{-N}\gamma$-pseudo-orbit. Then the sequence $\setb{y^{(i)}_N}_{i=0}^\infty$ is an asymptotic $\gamma$-pseudo-orbit, so let $z\in X$ be a point which asymptotically $\delta$-traces it. Define $q\in \inv{X}{f}$ by putting $q_i=f^{N-i}(z)$ for $0\leq i\leq N$, and for $i>N$ let $q_i$ be any point in $f^{-1}(q_{i-1})$. We claim that $q$ asymptotically $\eps$-traces $\set{y^{(i)}}_{i=0}^\infty$. Note that $\eps$-tracing is almost obvious, because for $n\leq N$ we have
		$$
		d(y^{(i)}_n,f^i(q_n))=d(f^{N-n}(y^{(i)}_N), f^{N-n}(f^i(q_N)))< \frac{\eps}{2N}
		$$
		and therefore $d(y^{(i)},\s^i(q))\leq \frac{\eps}{2}+\sum_{n=0}^N d(y^{(i)}_n,f^i(q_n))<\eps$.
		By the same argument $\lim_{i\to \infty}d(y^{(i)}_n,f^i(q_n))=0$ for every $n\leq N$.
		But if we fix any $n>N$ then for $i\ge n-N$ we have
		\begin{align*}
		d\big(y^{(i)}_n,f^i(q_n)\big) &= d\big(y^{(i)}_n,f^{i-(n-N)}(q_N)\big)\\
		&\leq \sum_{j=1}^{n-N} d\big(y^{(i-(j-1))}_{n-(j-1)},y^{(i-j)}_{n-j}\big) + d\big(y^{(i-(n-N))}_N,f^{i-(n-N)}(q_N)\big)\\
		&= \sum_{j=1}^{n-N} d\big(y^{(i-j+1)}_{n-j+1},f(y^{(i-j)}_{n-j+1})\big) + d\big(y^{(i-(n-N))}_N,f^{i-(n-N)}(q_N)\big)
		\end{align*}
		and so also in this case $\lim_{i\to \infty}d(y^{(i)}_n,f^i(q_n))=0$, where we use the fact that $z=q_N$ asymptotically traces $\setb{y_N^{(i)}}_{i=0}^\infty$ and that $\setb{y_{k}^{(i)}}_{i=0}^\infty$ is an asymptotic pseudo-orbit for each $k\ge 0$. Thus, $q$ asymptotically traces $\set{y^{(i)}}_{i=0}^\infty$, which completes the proof.
	\end{proof}
	
	\section*{Acknowledgements}
	
	
	This research was supported in part by a London Mathematical Society Research in Pairs - Scheme 4 grant which supported research visits of the authors.
	
	Chris Good and Mate Puljiz acknowledge support from the European Union through funding under FP7-ICT-2011-8 project HIERATIC (316705).
	Research of Piotr Oprocha was supported by National Science Centre, Poland (NCN), grant no. 2015/17/B/ST1/01259.
	Financial support of the above institutions is gratefully acknowledged.

\end{document}